\documentclass[12pt]{article}       

\usepackage{amsthm}
\usepackage{amsmath}
\usepackage{graphicx}
\usepackage{amssymb}
\usepackage{bold-extra}
\usepackage{amsthm}
\usepackage{fullpage}

\newtheorem{thm}{Theorem}[section]
\newtheorem{lma}[thm]{Lemma}
\newtheorem{cor}[thm]{Corollary}

\newtheorem{prop}[thm]{Proposition}

\newtheorem*{adef}{Definition}

\numberwithin{equation}{section}

\DeclareMathOperator{\StrEmb}{\#StrEmb}
\DeclareMathOperator{\ColStrEmb}{\#ColStrEmb}

\begin{document}

\newcommand{\leqfptT}{ $\leq^{\textrm{fpt}}_{\textrm{T}}$ }
\newcommand{\leqfptP}{ $\leq^{\textrm{fpt}}_{\textrm{pars}}$ }
\newcommand{\paramcount}[1]{\textup{\textbf{p-\#}}\textsc{#1}}
\newcommand{\paramdec}[1]{\textup{\textbf{p-}}\textsc{#1}}
\newcommand{\genprob}{Induced Subgraph With Property}
\newcommand{\genprobcol}{Multicolour Induced Subgraph with Property}

\title{The parameterised complexity of counting even and odd induced subgraphs \thanks{Research supported by EPSRC grant EP/I011935/1 (``Computational Counting'')}}
\date{\today}
\author{Mark Jerrum and Kitty Meeks\thanks{Current address: School of Mathematics and Statistics, University of Glasgow; \texttt{kitty.meeks@glasgow.ac.uk}} \\
\small{School of Mathematical Sciences, Queen Mary, University of London} \\ 
\texttt{\small{\{m.jerrum,k.meeks\}@qmul.ac.uk}}}
\maketitle

\begin{abstract}
We consider the problem of counting, in a given graph, the number of induced $k$-vertex subgraphs which have an even number of edges, and also the complementary problem of counting the $k$-vertex induced subgraphs having an odd number of edges.  We demonstrate that both problems are \#W[1]-hard when parameterised by $k$, in fact proving a somewhat stronger result about counting subgraphs with a property that only holds for some subset of $k$-vertex subgraphs which have an even (respectively odd) number of edges.  On the other hand, we show that each of the problems admits an FPTRAS.  These approximation schemes are based on a surprising structural result, which exploits ideas from Ramsey theory.
\end{abstract}

\section{Introduction}

In this paper we consider, from the point of view of parameterised complexity, the problems of counting the number of induced $k$-vertex subgraphs having an even (respectively odd) number of edges; while these two are clearly equivalent in terms of the complexity of exact counting, the existence of an approximation algorithm for one of these problems does not automatically imply that the other is approximable.  Formally, the problems we consider are defined as follows.
\\

\hangindent=1cm
\paramcount{Even Subgraph} \\
\textit{Input:} A graph $G = (V,E)$, and an integer $k$. \\
\textit{Parameter:} $k$. \\
\textit{Question:} How many induced $k$-vertex subgraphs of $G$ have an even number of edges? \\
\\

\hangindent=1cm
\paramcount{Odd Subgraph} \\
\textit{Input:} A graph $G = (V,E)$, and an integer $k$. \\
\textit{Parameter:} $k$. \\
\textit{Question:} How many induced $k$-vertex subgraphs of $G$ have an odd number of edges? \\

We shall refer to an induced subgraph having an even (respectively odd) number of edges as an even subgraph (respectively odd subgraph).

It has previously been observed by Goldberg et al.~\cite{goldberg10} that the related problem of counting all induced even subgraphs of a graph~$G$ (that is, the sum over all $k$ of the number of induced $k$-vertex subgraphs having an even number of edges) is polynomial-time solvable.  It was noted in \cite{goldberg10} that counting even subgraphs may be viewed as the evaluation of the partition function of a 2-spin system defined by a certain $2\times2$ matrix~$H_2$.  It transpires that the matrix $H_2$ corresponds to a computationally tractable partition function (Theorem 1.2 of \cite{goldberg10}, together with the comment immediately following).

To understand how this tractable case might arise, consider the following encoding of $n$-vertex graphs by quadratic polynomials over $\mathbb{F}_2$. Identifying the vertices of $G$ with the set $\{1,2,\ldots n\}$, introduce indeterminates $X_1,\ldots,X_n$ and encode each edge $ij$ of $G$ by the monomial $X_iX_j$.  The polynomial $p_G$ corresponding to graph $G$ is simply the sum $p_G(X_1,\ldots,X_n)=\sum_{ij\in E(G)}X_iX_j$ of monomials over all edges of~$G$.  The number of even subgraphs of $G$ is then equal to the number of solutions to the equation $p_G(X_1,\ldots, X_n)=0$.  Fortunately there is an efficient algorithm for computing the number of solutions to a quadratic polynomial over $\mathbb{F}_2$:  see Ehrenfeucht and Karpinski \cite{ehrenfeucht90}, or Theorems 6.30 and 6.32 of Lidl and Niederreiter \cite{lidl83}.

The techniques used to derive this result do not translate to the situation in which we specify the number of vertices in the desired subgraphs, and indeed we prove in this paper that, up to standard assumptions of parameterised complexity, \paramcount{Even Subgraph} and \paramcount{Odd Subgraph} cannot even be solved on an $n$-vertex graph in time $f(k)n^{O(1)}$, where $f$ is allowed to be an arbitrary computable function (note, however, that these problems can clearly be solved in time $f(k)n^k$ by exhaustive search, and so are polynomial-time solvable for any fixed $k$).  On the other hand, we show that both problems are efficiently approximable from the point of view of parameterised complexity.

Given the correspondence between graphs and quadratic polynomials, we can alternatively express our result as follows:  although the problem of counting \emph{all} solutions to a quadratic equation over $\mathbb{F}_2$ is tractable, the variant in which we restrict the count to solutions of a certain weight~$k$ is intractable in the sense of parameterised complexity.  Here ``weight'' is taken to be the number of variables set to~1.

The problems \paramcount{Even Subgraph} and \paramcount{Odd Subgraph} fall within the wider category of subgraph counting problems, which have received significant attention from the parameterised complexity community in recent years (see, for example \cite{arvind02,chen07,chen08,radu13,flum04,connected,bddlayers,treewidth}).  In particular, our hardness results complement a number of recent results \cite{radu14,bddlayers} which prove large families of such subgraph counting problems to be intractable from the point of view of parameterised complexity, making further progress towards a complete complexity classification of this type of parameterised counting problem.  The connections to previous work on subgraph counting will be discussed in more detail in Section \ref{previous} below.

In Section \ref{exact}, we prove hardness not only for the specific problems \paramcount{Even Subgraph} and \paramcount{Odd Subgraph}, but also for a much larger collection of problems in which the goal is to count even or odd subgraphs that also satisfy some arbitrary additional condition.  The proof of this result exploits the theory of combinatorial lattices to demonstrate the invertability of a relevant matrix.  These ideas were previously used in a similar way in \cite{connected}, but here we use a different method to construct the matrix, a strategy which is potentially applicable to other subgraph counting problems whose complexity is currently open.

The approximability results in Section \ref{approx} are based on a somewhat surprising structural result.  We demonstrate that any graph $G$ on $n$ vertices which contains at least one even (respectively odd) $k$-vertex subgraph, where $n$ is sufficiently large compared with $k$, must in fact contain a large number of such subgraphs; in particular, there will be sufficiently many of the desired subgraphs that a standard random sampling technique will provide a good estimate of the total number without requiring too many trials.  Moreover, we show that any (sufficiently large) $G$ that contains no even (respectively odd) $k$-vertex subgraph must belong to one of a small number of easily recognisable classes.  The proof of this structural result, which builds on Ramsey theoretic ideas, is very much specific to properties which depend only on a parity condition on the number of edges; however, it is nevertheless possible that similar results might hold for other subgraph counting problems, particularly those in which the desired property depends only on the number of edges present in the subgraph.

Before proving any of the results, we complete this section with a summary of the main notation used in the paper in Section \ref{notation}, a brief review of some of the key concepts from parameterised counting complexity in Section \ref{param}, and finally in Section \ref{previous} a discussion of the relationship of this paper to previous work.

\subsection{Notation and definitions}
\label{notation}

Throughout this paper, all graphs are assumed to be simple, that is, they do not have loops or multiple edges.  Given a graph $G = (V,E)$, and a subset $U \subseteq V$, we write $G[U]$ for the subgraph of $G$ induced by the vertices of $U$.  We denote by $e(G)$ the number of edges in $G$, and for any vertex $v \in V$ we write $N(v)$ for the set of neighbours of $v$ in $G$.  For any $k \in \mathbb{N}$, we write $[k]$ as shorthand for $\{1,\ldots,k\}$, and denote by $S_k$ the set of all permutations on $[k]$, that is, injective functions from $[k]$ to $[k]$.  We write $V^{(k)}$ for the set of all subsets of $V$ of size exactly $k$, and $V^{\underline{k}}$ for the set of $k$-tuples $(v_1,\ldots,v_k) \in V^k$ such that $v_1,\ldots,v_k$ are all distinct.   We denote by $\overline{G}$ the \emph{complement} of $G$, that is $\overline{G} = (V,E')$ where $E' = V^{(2)} \setminus E$.

If $G$ is coloured by some colouring $f: V \rightarrow [k]$ (not necessarily a proper colouring), we say that a subset $U \subseteq V$ is \emph{colourful} (under $f$) if, for every $i \in [k]$, there exists exactly one vertex $u \in U$ such that $f(u) = i$; note that this can only be achieved if $U \in V^{(k)}$.

We will be considering labelled graphs, where a labelled graph is a pair $(H, \pi)$ such that $H$ is a graph and $\pi : [|V(H)|] \rightarrow V(H)$ is a bijection.  We write $\mathcal{L}(k)$ for the set of all labelled graphs on $k$ vertices.  Given a graph $G = (V,E)$ and a $k$-tuple of vertices $(v_1,\ldots,v_k) \in V^{\underline{k}}$, $G[v_1,\ldots,v_k]$ denotes the labelled graph $(H,\pi)$ where $H = G[\{v_1,\ldots,v_k\}]$ and $\pi(i) = v_i$ for each $i \in [k]$.  

Given two graphs $G$ and $H$, a \emph{strong embedding} of $H$ in $G$ is an injective mapping $\theta: V(H) \rightarrow V(G)$ such that, for any $u,v \in V(H)$, $\theta(u)\theta(v) \in E(G)$ if and only if $uv \in E(H)$.  We denote by $\StrEmb(H,G)$ the number of strong embeddings of $H$ in $G$.  If $\mathcal{H}$ is a class of labelled graphs on $k$ vertices, we set 
\begin{align*}
\StrEmb(\mathcal{H},G) =  \qquad \qquad  & \\
|\{\theta: [k] \rightarrow V(G) \quad : \quad & \theta \text{ is injective and } \exists (H,\pi) \in \mathcal{H}  \text{ such that } \\ &  \theta(i)\theta(j) \in E(G) \iff \pi(i)\pi(j) \in E(H)\}|.
\end{align*}
If $G$ is also equipped with a $k$-colouring $f$, where $|V(H)| = k$, we write $\ColStrEmb(H,G,f)$ for the number of strong embeddings of $H$ in $G$ such that the image of $V(H)$ is colourful under $f$.  Similarly, we set 
\begin{align*}
\ColStrEmb(\mathcal{H},G,f) = \qquad \qquad & \\
 |\{\theta:[k] \rightarrow V(G) \quad : \quad & \theta \text{ is injective, } \exists (H,\pi) \in \mathcal{H} \text{ such that } \\ 
 						& \theta(i)\theta(j) \in E(G) \iff \pi(i)\pi(j) \in E(H), \\
 						& \text{and $\theta([k])$ is colourful under } f\}|.
\end{align*}

\subsection{Parameterised counting complexity}
\label{param}

In this section, we recall some key notions from parameterised counting complexity which will be used in the rest of the paper.  Let $\Sigma$ be a finite alphabet.  A parameterised counting problem is a pair $(\Pi,\kappa)$, where $\Pi: \Sigma^* \rightarrow \mathbb{N}_0$ is a function  and $\kappa: \Sigma^* \rightarrow \mathbb{N}$ is a parameterisation (a polynomial-time computable mapping).  An algorithm $A$ for a parameterised counting problem $(\Pi,\kappa)$ is said to be an \emph{fpt-algorithm} if there exists a computable function $f$ and a constant $c$ such that the running time of $A$ on input $I$ is bounded by $f(\kappa(I))|I|^c$.  Problems admitting an fpt-algorithm are said to belong to the class FPT.

To understand the complexity of parameterised counting problems, Flum and Grohe \cite{flum04} introduce two kinds of reductions between such problems; we shall make use of so-called fpt Turing reductions.

\begin{adef}
 An fpt Turing reduction from $(\Pi,\kappa)$ to $(\Pi',\kappa')$ is an algorithm $A$ with an oracle to $\Pi'$ such that
\begin{enumerate}
\item $A$ computes $\Pi$,
\item $A$ is an fpt-algorithm with respect to $\kappa$, and
\item there is a computable function $g:\mathbb{N} \rightarrow \mathbb{N}$ such that for all oracle queries ``$\Pi'(y) = \; ?$'' posed by $A$ on input $x$ we have $\kappa'(I') \leq g(\kappa(I))$.
\end{enumerate}
In this case we write $(\Pi,\kappa)$ \leqfptT $(\Pi',\kappa')$.
\end{adef}

Using these notions, Flum and Grohe introduce a hierarchy of parameterised counting complexity classes, \#W[$t$], for $t \geq 1$ (see \cite{flum04,flumgrohe} for the formal definition of these classes).  Just as it is considered to be very unlikely that W[1] = FPT, it is unlikely that there exists an algorithm running in time $f(k)n^{O(1)}$ for any problem that is hard for the class \#W[1] under fpt Turing reductions.  One useful \#W[1]-complete problem, which we will use in our reductions, is the following:
\\

\hangindent=1cm
\paramcount{Multicolour Clique} \\
\textit{Input:} A graph $G = (V,E)$, and a $k$-colouring $f$ of $G$. \\
\textit{Parameter:} $k$. \\
\textit{Question:} How many $k$-vertex cliques in $G$ are colourful under $f$? \\

This problem can easily be shown to be \#W[1]-hard (along the same lines as the proof of the W[1]-hardness of \paramdec{Multicolour Clique} in \cite{fellows09}) by means of a reduction from \paramcount{Clique}, shown to be \#W[1]-hard in \cite{flum04}.

When considering approximation algorithms for parameterised counting problems, an ``efficient'' approximation scheme is an FPTRAS, as introduced by Arvind and Raman \cite{arvind02}; this is the parameterised analogue of an FPRAS (fully polynomial randomised approximation scheme).
\begin{adef}
An FPTRAS for a parameterised counting problem $(\Pi,\kappa)$ is a randomised approximation scheme that takes an instance $I$ of $\Pi$ (with $|I| = n$), and real numbers $\epsilon > 0$ and $0 < \delta < 1$, and in time $f(\kappa(I)) \cdot g(n,1/\epsilon,\log(1/\delta))$ (where $f$ is any function, and $g$ is a polynomial in $n$, $1/\epsilon$ and $\log(1 / \delta)$) outputs a rational number $z$ such that
$$\mathbb{P}[(1-\epsilon)\Pi(I) \leq z \leq (1 + \epsilon)\Pi(I)] \geq 1 - \delta.$$
\end{adef}

\subsection{Relationship to previous work}
\label{previous}

The problems \paramcount{Even Subgraph} and \paramcount{Odd Subgraph} defined above belong to the family of subgraph counting problems, \paramcount{\genprob}$(\Phi)$, introduced in \cite{connected} and studied further in \cite{bddlayers,treewidth}.  This general problem is defined as follows, where $\Phi$ is a family $(\phi_1,\phi_2,\ldots)$ of functions $\phi_k: \mathcal{L}(k) \rightarrow \{0,1\}$ such that the function mapping $k \mapsto \phi_k$ is computable.
\\

\hangindent=1cm
\paramcount{\genprob}($\Phi$) (\paramcount{ISWP}$(\Phi)$)\\
\textit{Input:} A graph $G = (V,E)$ and an integer $k$.\\
\textit{Parameter:} $k$. \\
\textit{Question:} What is the cardinality of the set $\{(v_1,\ldots,v_k) \in V^{\underline{k}}: \phi_k(G[v_1,\ldots,v_k]) = 1 \}$? \\

For any $k$, we write $\mathcal{H}_{\phi_k}$ for the set $\{(H,\pi) \in \mathcal{L}(k): \phi_k(H,\pi) = 1\}$, and set $\mathcal{H}_{\Phi} = \bigcup_{k \in \mathbb{N}} \mathcal{H}_{\phi_k}$.  We can equivalently regard the problem as that of counting induced labelled $k$-vertex subgraphs that belong to $\mathcal{H}_{\Phi}$.

For the problems considered in this paper, we define $\phi_k$ so that $\phi_k(H,\pi) = 1$ if and only if the number of edges in $H$ is even (for the case of \paramcount{Even Subgraph}) or odd (for the case of \paramcount{Odd Subgraph}).  It is clear that these problems are \emph{symmetric}, that is they do not depend on the labelling of the vertices, and so it would be possible to consider unlabelled rather than labelled graphs.  Nevertheless, we choose to use the framework of this general model which encompasses problems such as \paramcount{Path} and \paramcount{Matching} (considered in \cite{flum04} and \cite{radu13} respectively) rather than just symmetric properties; see \cite{bddlayers} for a more detailed discussion of how such problems can be defined in this model).  All problems of this form were shown to belong to the class \#W[1] in \cite{connected}.

\begin{prop}[\cite{connected}]
For any $\Phi$, the problem \paramcount{ISWP}($\Phi$) belongs to \#W[1].
\label{in-W}
\end{prop}

It is sometimes useful to consider the multicolour version of this general problem (and specific examples falling within it, such as \paramcount{Multicolour Clique} defined above), in which we only wish to count colourful tuples having our property; this version of the problem is defined as follows.
\\

\hangindent=1cm
\paramcount{\genprobcol}($\Phi$) (\paramcount{MISWP}$(\Phi)$) \\
\textit{Input:} A graph $G = (V,E)$, an integer $k$ and colouring $f: V \rightarrow [k]$.\\
\textit{Parameter:} $k$. \\
\textit{Question:} What is the cardinality of the set $\{(v_1,\ldots,v_k) \in V^{\underline{k}}: \phi_k(G[v_1,\ldots,v_k]) = 1$ and $\{v_1,\ldots,v_k\}$ is colourful with respect to $f$ $\}$? \\

The complexities of the multicolour and uncoloured versions of the problem, for any given $\Phi$, are related in the following way.

\begin{lma}[\cite{bddlayers}]
For any family $\Phi$, we have \paramcount{MISWP}$(\Phi)$ \leqfptT \paramcount{ISWP}$(\Phi)$.
\label{uncol-col}
\end{lma}

In Section \ref{exact}, we prove that \paramcount{ISWP}$(\Phi)$ is \#W[1]-complete in the case that, for each $k$, $\phi_k(H,\pi) = 1$ only when the number of edges in $H$ has a specified parity; this clearly imples the hardness of \paramcount{Even Subgraph} and \paramcount{Odd Subgraph}, but does not rely on the additional assumption that the property depends only on the parity of the number of edges in $H$.

This hardness result adds to the growing list of situations in which \paramcount{ISWP}$(\Phi)$ is known to be \#W[1]-complete.  It is known that any of the following conditions on $\Phi$ is sufficient to imply \#W[1]-completeness:
\begin{itemize}
\item the number of distinct edge-densities of graphs $(H,\pi)$ with $\phi_k(H,\pi) = 1$ is $o(k^2)$ \cite{bddlayers};
\item $\phi_k(H,\pi)$ is true if and only if the number of edges in $H$ belongs to an interval from some collection $I_k$, with $|I_k| = o(k^2)$ \cite{bddlayers};
\item there is no constant $t$ such that, for any $k$, every minimal graph $H$ (with respect to inclusion) such that $\phi_k(H,\pi) = 1$ for some labelling $\pi$, $H$ has treewidth at most $t$ \cite{treewidth};
\item $\phi_k(H,\pi) = 1$ for exactly one fixed $(H,\pi)$ \cite{chen08}.
\end{itemize}
Recently, Curticapean and Marx \cite{radu14} proved a dichotomy result for the sub-family of these problems in which $\phi_k(H,\pi) = 1$ if and only if $(H,\pi)$ contains a fixed labelled graph $(H_k,\pi_k)$ as a labelled subgraph, showing that this problem is in FPT whenever the vertex cover number of $H_k$ is bounded by some constant for all $k$, and otherwise is \#W[1]-complete.  Indeed, to the best of our knowledge, there is not known to be any property $\Phi$ such that \paramcount{ISWP}$(\Phi)$ belongs to FPT other than the ``trivial" properties for which there exists a fixed integer $t$ so that, for every $k$, it is possible to determine whether $\phi_k(H,\pi) = 1$ by examining only edges incident with some subset of at most $t$ vertices.

\section{Exact counting}
\label{exact}

In this section we will prove the \#W[1]-completeness of both \paramcount{Even Subgraph} and \paramcount{Odd Subgraph}; both will follow from a more general result concerning the family of problems described in Section \ref{previous} above.  The theory of lattices is crucial to the reduction we give in Section \ref{reduction}, and we begin in Section \ref{lattices} by recalling the key facts we will need.

\subsection{The subset lattice}
\label{lattices}

In the proof of Theorem \ref{parity-hard} below, we will need to consider the lattice formed by subsets of a finite set, with a partial order given by subset inclusion.  Here we recall some existing results about lattices on posets.

A lattice is a partially ordered set $(P,\leq)$ satisfying the condition that, for any two elements $x,y \in P$, both the \emph{meet} and \emph{join} of $x$ and $y$ also belong to $P$, where the meet of $x$ and $y$, written $x \wedge y$, is defined to be the unique element $z$ such that 
\begin{enumerate}
\item $z \leq x$ and $z \leq y$, and
\item for any $w$ such that $w \leq x$ and $w \leq y$, we have $w \leq z$,
\end{enumerate}
and the join of $x$ and $y$, $x \vee y$, is correspondingly defined to be the unique element $z'$ such that 
\begin{enumerate}
\item $x \leq z'$ and $y \leq z'$, and
\item for any $w$ such that $x \leq w$ and $y \leq w$, we have $z' \leq w$.
\end{enumerate}

Note therefore that, if $x$ and $y$ are two elements of the \emph{subset lattice} formed by all subsets of a finite set, ordered by inclusion, then $x \wedge y = x \cap y$ and $x \vee y = x \cup y$.

In the proof of Theorem \ref{parity-hard}, we will consider a so-called \emph{meet-matrix} on a subset lattice: if $S = \{x_1,\ldots,x_n\} \subseteq P$, and $f:P \rightarrow \mathbb{C}$ is a function, then the meet-matrix on $S$ with respect to $f$ is the matrix $A = (a_{ij})_{1\leq i,j \leq n}$ where $a_{ij} = f(x_i \wedge x_j)$.  Explicit formulae are known for the determinant of the meet-matrix of the entire lattice (or a meet-closed subset of it) \cite[Corollary 2]{haukkanen96}, and such results were exploited to prove the hardness of \paramcount{Connected Induced Subgraph} \cite{connected}.  However, for the problems considered here, we are not able to apply this approach directly, since the determinant of the meet-matrix of the entire lattice may well be zero.  Instead, we restrict attention to a sub-matrix corresponding to a suitable subset of the lattice (that is not meet-closed); the remainder of this section is concerned with computing the determinant of this sub-matrix.

A decomposition result for such matrices is given in \cite{bhat91}, where $\Psi_f$ denotes the generalised Euler totient function, defined inductively by 
$$\Psi_f(x_i) = f(x_i) - \sum_{\substack{x_j \leq x_i \\ x_j \in P \setminus \{x_i\}}} \Psi_f(x_j).$$
\begin{thm}[\cite{bhat91}, special case of Theorem 12]
Let $S = \{x_1,\ldots,x_n\}$ be a subset of the finite lattice $(P,\leq)$, where $P = \{x_1,\ldots,x_n,x_{n+1},\ldots,x_m\}$, let $f:P \rightarrow \mathbb{R}$ be a function, and let $A = (a_{ij})_{1 \leq i,j \leq n}$ be the matrix given by $a_{ij} = f(x_i \wedge x_j)$.  Then
$$A = E \Lambda E^{T},$$
where $E = (e_{ij})_{\substack{1 \leq i \leq n \\ 1 \leq j \leq m}}$ is the matrix given by
\begin{equation*}
e_{ij} = \begin{cases}
			1	& \text{if } x_j \leq x_i \\
			0	& \text{otherwise,}
		 \end{cases}
\end{equation*}
and $\Lambda$ is the $m \times m$ diagonal matrix whose $r^{th}$ diagonal entry is equal to $\Psi_f(x_r)$.
\label{bhat-decomp}
\end{thm}

It will be convenient to express the function $\Psi_f$ in terms of the M\"{o}bius function $\mu$ on a poset, which is defined by
\begin{equation*}
\mu(x,y) = \begin{cases}
				1	& \text{if } x = y \\
				- \sum_{z \in P, x \leq z < y} \mu(x,z) & \text{for } x < y \\
				0	& \text{otherwise}.
			\end{cases}
\end{equation*}
It is well-known (and easily verified from the definition above) that, for two elements $x$ and $y$ of a subset lattice with $x \leq y$, we have $\mu(x,y) = (-1)^{|y \setminus x|} = (-1)^{|y| - |x|}$.  To express the totient function in terms of the M\"{o}bius function, we use the following observation of Haukkanen:
\begin{prop}[\cite{haukkanen96}, Example 2]
$$\Psi_f(x_i) = \sum_{\substack{x_j \in P \\ x_j \leq x_i}} f(x_j)\mu(x_j,x_i).$$
\end{prop}

We now derive an expression for the determinant of a meet-matrix when certain additional conditions are also satisfied.

\begin{cor}
Let $(P,\leq)$ be a finite lattice and $f: P \rightarrow \mathbb{R}$ a function.  Set $S$ to be the upward closure of the support of $f$, that is,
$$S = \{x \in P: \exists y \in P \text{ with } y \leq x \text{ and } f(y) \neq 0 \},$$
and suppose that $S = \{x_1,\ldots,x_n\}$.  Let $A = (a_{ij})_{1 \leq i,j \leq n}$ be the matrix given by $a_{ij} = f(x_i \wedge x_j)$.  Then
$$\det(A) = \prod_{i=1}^n \sum_{\substack{x_j \in S \\ x_j \leq x_i}} f(x_j) \mu(x_j,x_i).$$
\label{determinant}
\end{cor}
\begin{proof}
Note that we may assume, without loss of generality, that the elements of $S$ are ordered so that, whenever $x_i < x_j$, we have $i<j$.

We know from Theorem \ref{bhat-decomp} that, if $P = \{x_1,\ldots,x_m\}$, then 
$$\det(A) = \det(E \Lambda E^{T}),$$
where $E = (e_{ij})_{\substack{1 \leq i \leq n \\ 1 \leq j \leq m}}$, with 
\begin{equation*}
e_{ij} = \begin{cases}
			1	& \text{if $x_j \leq x_i$} \\
			0	& \text{otherwise,}
		 \end{cases}
\end{equation*}		
and $\Lambda$ is the $m \times m$ diagonal matrix with its $r^{th}$ diagonal entry equal to 
$$\Psi_f(x_r) = \sum_{\substack{x_j \in P \\ x_j \leq x_r}} f(x_j)\mu(x_j,x_r).$$

Observe first that, by definition of $S$, we must have $\Psi_f(x_r) = 0$ for all $x_r \notin S$, and moreover for $x_r \in S$ we have 
$$\Psi_f(x_r) = \sum_{\substack{x_j \in S \\ x_j \leq x_r}} f(x_j) \mu(x_j,x_r).$$
Thus the only non-zero entries of $\Lambda$ are in the $n \times n$ submatrix containing entries belonging to the first $n$ rows and first $n$ columns; we denote this submatrix $\Lambda'$.  It therefore follows that
$$E \Lambda E^{T} = E' \Lambda' (E')^{T},$$
where $E'$ denotes the $n \times n$ submatrix of $E$ obtained by taking the first $n$ columns of $E$.  We further note that, by our choice of ordering of $S$, we have $e_{ij} = 0$ for $1 \leq i < j \leq n$, and moreover, by definition of $E$, we have $e_{ii} = 1$ for $1 \leq i \leq n$.  Thus it is clear that $E'$ is a triangular matrix, and so its determinant is the product of its diagonal entries, which are all equal to one.  Hence we have $\det(E') = \det((E')^{T}) = 1$.  This implies that
$$\det(A) = \det(E' \Lambda' (E')^{T}) = \det(E') \det(\Lambda') \det((E')^{T}) = \det(\Lambda').$$
Since $\Lambda'$ is a diagonal matrix, its determinant is simply the product of its diagonal entries.  Hence
$$\det(A) = \prod_{i=1}^n \sum_{x_j \leq x_i} f(x_j)\mu(x_j,x_i),$$
as required.
\end{proof}

\subsection{The reduction}
\label{reduction}

We now prove our general hardness result, demonstrating that it is \#W[1]-hard to count $k$-vertex induced subgraphs having a property that is only true for subgraphs whose number of edges has a given parity.

\begin{thm}
Let $\Phi = (\phi_1,\phi_2,\ldots)$ be a family of functions $\phi_k: \mathcal{L}(k) \rightarrow \{0,1\}$, infinitely many of which are not identically zero, and such that the mapping $k \mapsto \phi_k$ is computable.  For each $k$, set $D_k = \{|E(H)|: \exists \pi \in S_k$ such that $(H,\pi) \in \mathcal{H}_{\phi_k}\}$, and suppose that, for each $k \in \mathbb{N}$, either $D_k \subseteq 2 \mathbb{N}$, or else $D_k \cap 2 \mathbb{N} = \emptyset$ (where $2 \mathbb{N}$ denotes the set of even, non-negative integers).  Then \paramcount{MISWP}$(\Phi)$ and \paramcount{ISWP}$(\Phi)$ are both \#W[1]-complete under fpt-Turing reductions.
\label{parity-hard}
\end{thm}
\begin{proof}
Inclusion in \#W[1] is an immediate consequence of Proposition \ref{in-W}.  To prove hardness, we give an fpt Turing reduction from the \#W[1]-complete problem \paramcount{Multicolour Clique} to \paramcount{MISWP}$(\Phi)$; hardness of \paramcount{ISWP}$(\Phi)$ then follows from Proposition \ref{uncol-col}.  Suppose that $G = (V,E)$ with colouring $f: V \rightarrow [k]$ is the input to an instance of \paramcount{Multicolour Clique}.  If $\phi_k \equiv 0$, we can (by the assumption that $\phi_{k'}$ is not identically zero for infinitely many values of $k'$) choose $k' > k$ such that $\phi_{k'} \not\equiv 0$; it is clear in this case that the graph $G' = (V',E')$ with colouring $f'$, where $V' = V \cup \{w_{k+1},\ldots,w_{k'}\}$, $E' = E \cup \{w_iv: v \in V' \setminus \{w_i\}, k+1 \leq i \leq k'\}$ and
\begin{equation*}
f'(v) = \begin{cases}
			f(v)	& \text{if } v \in V \\
			i		& \text{if } v = w_i,
		\end{cases}
\end{equation*}
contains a multicolour clique (on $k'$ vertices) if and only if $G$ with colouring $f$ contains a multicolour clique on $k$ vertices.  Thus we may assume without loss of generality that $\phi_k$ is not identically zero.

For any subset $I \subseteq [k]^{(2)}$, we now define a graph $H_I = ([k],I)$.  We define a collection of such subsets, $\mathcal{I}$, by setting
$$\mathcal{I} = \{I \subseteq [k]^{(2)}: \exists I' \subseteq I \text{ and $\pi \in S_k$ such that } \phi_k(H_{I'},\pi) = 1\}.$$
Note that, by the assumption that $\phi_k$ is not identically zero, $\mathcal{I} \neq \emptyset$; moreover, we must have $[k]^{(2)} \in \mathcal{I}$.   Let $I_1,\ldots,I_m$ be a fixed enumeration of $\mathcal{I}$, with subsets in non-decreasing order of cardinality, and note therefore that $I_m = [k]^{(2)}$.  For each $i \in [m]$, we set $G_i = (V,E_i)$ where $E_i = \{uv \in E: \{f(u),f(v)\} \in I_i\}$, and set $z_i = \ColStrEmb(\mathcal{H}_{\phi_k},G_i,f)$.  

Additionally, we associate with each colourful subset $U \subset V$ a subset of $[k]^{(2)}$, setting $I(U) = \{\{f(u),f(w)\}: uw \in E(G[U])\}$.  For each $i \in [m]$, we denote by $N_i$ the number of colourful subsets $U \subset V$ such that $I(U) = I_i$; observe that the number of colourful cliques in $G$ with respect to $f$ is then equal to $N_m$.

We now define a matrix $A = (a_{ij})_{i,j = 1}^{m}$ by setting
\begin{equation*}
a_{ij} = \sum_{\pi \in S_k} \phi_k(H_{I_i \cap I_j},\pi).
\end{equation*}		 
We claim that, with this definition, 
$$ A \cdot \mathbf{N} = \mathbf{z},$$
where $\mathbf{N} = (N_1,\ldots,N_{m})^T$ and $\mathbf{z} = (z_{1},\ldots,z_{m})^T$, with $z_i = \ColStrEmb(\mathcal{H}_{\phi_k},G_i,f)$.  To see that this is true, observe that, for each $i \in [m]$,
\begin{align*}
z_i & = \ColStrEmb(\mathcal{H}_{\phi_k},G_i,f) \\
	& = \sum_{\substack{(v_1,\ldots,v_k) \in V^{\underline{k}} \\ \text{$\{v_1,\ldots,v_k\}$ colourful}}} \phi_k(G_i[v_1,\ldots,v_k]) \\
	& = \sum_{\substack{\{v_1,\ldots,v_k\} \in V^{(k)} \\ \text{$\{v_1,\ldots,v_k\}$ colourful}}} \sum_{\pi \in S_k} \phi_k(G_i[v_{\pi(1)},\ldots,v_{\pi(k)}]) \\
	& = \sum_{j=1}^m \sum_{\substack{\{v_1,\ldots,v_k\} \in V^{(k)} \\ \text{$\{v_1,\ldots,v_k\}$ colourful} \\ I(\{v_1,\ldots,v_k\}) = I_j}} \sum_{\pi \in S_k} \phi_k(G_i[v_{\pi(1)},\ldots,v_{\pi(k)}]) \\
	& = \sum_{j = 1}^m N_j \sum_{\pi \in S_k} \phi_k(H_{I_i \cap I_j},\pi) \\
	& = \sum_{j=1}^m N_j a_{ij}, 
\end{align*}
as required.

Moreover, if $\wedge$ denotes the meet operation on the subset lattice ordered by subset inclusion (so this is in fact set intersection), and $g: \mathcal{I} \rightarrow \mathbb{N}$ is defined by $g(I) = \sum_{\pi \in S_k} \phi_k (H_I, \pi)$, we see that 
$$a_{ij} = g(I_i \wedge I_j).$$
It is therefore clear that $A$ satisfies the premise of Corollary \ref{determinant} (where $\mathcal{I}$ and $g$ take the roles of $S$ and $f$ respectively) and so the conclusion tells us that 
\begin{align*}
\det(A) & = \prod_{i=1}^m \sum_{I_j \leq I_i} g(I_j)\mu(I_j,I_i) \\
		& = \prod_{i=1}^m \sum_{I_j \subseteq I_i} (-1)^{|I_i| - |I_j|} g(I_j) \\
\end{align*}		
Considering one of these factors, $\sum_{I_j \subseteq I_i} (-1)^{|I_i| - |I_j|} g(I_j)$, observe that, by definition of $\mathcal{I}$, there exists at least one $I_j \subseteq I_i$ such that $g(I_j) \neq 0$ (and note that $g(I_j) \geq 0$ for all $I_j \in \mathcal{I}$).  We know that either $D_k \subseteq 2\mathbb{N}$ or $D_k \cap 2\mathbb{N} = \emptyset$: in the former case, for every $I_j \subseteq I_i$ such that $g(I_j) \neq 0$ we have $|I_j|$ even, and in the latter we have $|I_j|$ odd for every such $I_j$.  Thus, in either case, $(-1)^{|I_i|-|I_j|}$ takes the same value for every $I_j \subseteq I_i$ with $g(I_j) \neq 0$, so all non-zero terms in the sum (of which we have already observed there must be at least one) have the same sign, guaranteeing that $\sum_{I_j \subseteq I_i} (-1)^{|I_i| - |I_j|} g(I_j) \neq 0$.  Since this holds for every $i \in [m]$, we see that 
$$\det(A) = \prod_{i=1}^m \sum_{I_j \subseteq I_i} (-1)^{|I_i| - |I_j|} g(I_j) \neq 0,$$
implying that $A$ is non-singular.

Observe that we can determine the value of $z_i = \ColStrEmb(\mathcal{H}_{\phi_k},G,f)$, for each $i \in [m]$, with a single call to an oracle for \paramcount{MISWP}$(\Phi)$ on the input $(G_i,f)$ (where the parameter value is unchanged); thus, with the use of such an oracle, we can determine in time $O(mn^2)$ (the quadratic time required to construct each graph $G_I$) the precise value of $\mathbf{z}$.  By non-singularity of $A$, we can then, in polynomial time, compute all values of $N_i$ for $i \in [m]$, and in particular the value of $N_{m}$, which is precisely the number of multicolour cliques in $G$ under the colouring $f$.  This gives the required fpt Turing reduction from \paramcount{Multicolour Clique} to \paramcount{MISWP}$(\Phi)$.
\end{proof}

The hardness of \paramcount{Even Subgraph} and \paramcount{Odd Subgraph} now follows immediately.

\begin{cor}
\paramcount{Even Subgraph} and \paramcount{Odd Subgraph} are both \#W[1]-complete under fpt-Turing reductions.
\label{even-odd-hard}
\end{cor}

\section{Decision and approximate counting}
\label{approx}

In contrast with the hardness results in Section \ref{exact} above, we now demonstrate that \paramcount{Even Subgraph} and \paramcount{Odd Subgraph} are efficiently approximable from the point of view of parameterised complexity, and also that the corresponding decision problems belong to FPT.  We begin in Section \ref{approx-prelim} with some preliminary facts we will use later in the section, before proving our key structural results in Section \ref{structural} and finally deriving the algorithmic implications of these results in Section \ref{alg-imp}.

\subsection{Background}
\label{approx-prelim}

Here we outline some of the background results that will be needed later in the section. 

We begin with several facts from Ramsey Theory, which will play an important role in our structural results below.  First, we need the following bound on Ramsey numbers which follows immediately from a result of Erd\H{o}s and Szekeres \cite{erdos-szekeres}:
\begin{thm}
Let $k \in \mathbb{N}$.  Then there exists $R(k) < 2^{2k}$ such that any graph on $n \geq R(k)$ vertices contains either a clique or an independent set on $k$ vertices.
\label{ramsey}
\end{thm}

We will also use the following easy corollary of this result, proved in \cite{bddlayers}.
\begin{cor}[\cite{bddlayers}, Corollary 1.2]
Let $G = (V,E)$ be an $n$-vertex graph, where $n \geq 2^{2k}$.  Then the number of $k$-vertex subsets $U \subset V$ such that $U$ induces either a clique or independent set in $G$ is at least
$$\frac{(2^{2k} - k)!}{(2^{2k})!}\frac{n!}{(n-k)!}.$$
\label{ramsey-cor}
\end{cor}

To simplify calculations in Section \ref{structural}, we will make use of the following well-known bounds on binomial coefficients:
$$\left( \frac{n}{k} \right)^{k} \leq \binom{n}{k} \leq \left(\frac{en}{k}\right)^k.$$
In particular, this tells us that, for $k \geq 3$,
\begin{equation}
\frac{(2^{2k} - k)!}{(2^{2k})!}\frac{n!}{(n-k)!} = \frac{(2^{2k}-k)!k!}{(2^{2k})!}\binom{n}{k} = \frac{\binom{n}{k}}{\binom{2^{2k}}{k}} \geq \frac{k^k}{\mathrm{e}^k2^{2k^2}}\binom{n}{k} > \frac{1}{2^{2k^2}}\binom{n}{k}
\label{ramsey-cor-bound}
\end{equation}

Finally, in Section \ref{alg-imp}, we will exploit the fact that, if we can guarantee that the proportion of $k$-vertex labelled subgraphs of a graph $G$ having the desired property is sufficiently large, we can make use of standard random sampling techniques to approximate the number of subgraphs having our property.  We will combine the following lemma, proved in \cite{treewidth}, with our structural results in the following section to demonstrate the existence of an FPTRAS for each of \paramcount{Even Subgraph} and \paramcount{Odd Subgraph}.  

\begin{lma}[\cite{treewidth}, Lemma 3.4]
Let $G=(V,E)$ be a graph on $n$ vertices and $\phi_k$ a mapping from labelled $k$-vertex graphs to $\{0,1\}$, and set $N_k(G)$ to be the number of $k$-tuples of vertices $(v_1,\ldots,v_k) \in V^{\underline{k}}$ satisfying $\phi_k(G[v_1,\ldots,v_k]) = 1$.  Suppose that there exists a polynomial $q(n)$ and a computable function $g(k)$ such that either $N_k(G)=0$ or $N_k(G) \geq \frac{1}{g(k)q(n)} \frac{n!}{(n-k)!}$, for all $k$ and $G$.  Then, for every $\epsilon > 0$ and $\delta \in (0,1)$ there is an explicit randomised algorithm which outputs an integer $\alpha$, such that
$$\mathbb{P}[|\alpha - N_k(G)| \leq \epsilon \cdot N_k(G)] \geq 1 - \delta,$$
and runs in time at most $g(k)\tilde{q}(n,\epsilon^{-1},\log(\delta^{-1}))$, where $g$ is a computable function and $\tilde{q}$ is a polynomial.
\label{exists-fptras}
\end{lma}

\subsection{Structural results}
\label{structural}

In this section we prove the key structural results which give rise to the algorithms described in Section \ref{alg-imp} below.  The key result from this section is that graphs (on sufficiently many vertices) containing no even $k$-vertex subgraph must belong to one of a small number of easily recognisable families of graphs, and a corresponding result holds for graphs containing no odd $k$-vertex subgraphs; moreover, any sufficiently large graph that contains at least one even $k$-vertex subgraph (respectively odd $k$-vertex subgraph) must in fact contain a large number of such subgraphs.

We begin with an easy condition which guarantees the existence of a $k$-vertex subgraph which shares all but one of its vertices with a $k$-clique but whose number of edges differs in parity from a $k$-clique.  Recall that $N(v)$ denotes the set of vertices adjacent to $v$.

\begin{prop}
Let $G$ be a graph which contains a $k$-vertex clique $H$ (where $k \geq 3$), and suppose there is a vertex $v \in V(G)$ such that $\emptyset \neq N(v) \cap V(H) \neq V(H)$.  Then $G$ contains a $k$-vertex induced subgraph $\widetilde{H}$, with $|V(H) \cap V(\widetilde{H})| = k - 1$, such that $e(H) \not\equiv e(\widetilde{H}) \pmod{2}$.
\label{even-odd-vertex}
\end{prop}
\begin{proof}
We denote by $r$ the number of non-neighbours of $v$ in $H$, that is $r = |V(H) \setminus N(v)|$, and recall that by assumption we have $0 < r < k$.  Thus there exists some vertex $u \in V(H) \setminus N(v)$, and some vertex $w \in V(H) \cap N(v)$.  Set $H_u = G[(V(H) \setminus u) \cup \{v\}]$ and $H_w = G[(V(H) \setminus w) \cup \{v\}]$, and observe that
$$e(H_u) = \binom{k}{2} - (r-1),$$
while
$$e(H_w) = \binom{k}{2} - r.$$
Thus, $e(H_u) = e(H_w) + 1$, so in particular $e(H_u) \not\equiv e(H_w) \pmod{2}$, and precisely one of these subgraphs will be the required subgraph $\widetilde{H}$.
\end{proof}

Under the additional assumption that $k$ is even, we can strengthen this result further.

\begin{cor}
Let $G$ be a graph which contains a $k$-vertex clique $H$, where $k \geq 4$ is even, and suppose that there is a vertex $v \in V(G)$ such that $N(v) \cap V(H) \neq V(H)$.  Then $G$ contains a $k$-vertex induced subgraph $\widetilde{H}$, with $|V(H) \cap V(\widetilde{H})| = k - 1$, such that $e(H) \not\equiv e(\widetilde{H}) \pmod{2}$.
\label{cong2-iso-vx}
\end{cor}
\begin{proof}
If in fact $N(v) \cap V(H) \neq \emptyset$, then we are done by Proposition \ref{even-odd-vertex}.  So suppose that $v$ has no neighbour in $H$.  But then the subgraph induced by $v$ together with any $k-1$ vertices of $H$ will have $\binom{k}{2} - (k-1)$ edges which, as $k$ is even, differs in parity from $\binom{k}{2}$.
\end{proof}

We now use these facts to characterise the situations in which a graph $G$ which contains a $k$-clique does not contain any $k$-vertex induced subgraph whose number of edges differs in parity from $\binom{k}{2}$.

\begin{lma}
Let $G$ be a graph which contains a clique on $k \geq 2$ vertices.  Then $G$ also contains a $k$-vertex subgraph $H$ such that $\binom{k}{2} \not\equiv e(H) \pmod{2}$, unless either
\begin{enumerate}
\item $G$ is a clique, or
\item $k$ is odd and $G$ is the disjoint union of two cliques.
\end{enumerate}
If either of these conditions holds, then every $k$-vertex subgraph $H$ of $G$ satisfies $e(H) \equiv \binom{k}{2} \pmod{2}$.
\label{different-parity}
\end{lma}
\begin{proof}
Note that the result is trivially true for $k=2$, so we shall assume that $k \geq 3$.  Suppose that $G$ contains no $k$-vertex subgraph $H$ such that $\binom{k}{2} \not\equiv e(H) \pmod{2}$.  Let $H'$ be a maximal clique in $G$, so certainly $|V(H')| \geq k$.  If in fact $H' = G$ then we are done, so assume that this is not the case.  Thus, by maximality of $H'$, for every vertex $v \in V(G)\setminus V(H')$ there exists some $u_v \in V(H')$ with $vu_v \notin E(G)$.  Note that, if $k$ is even, it follows from Corollary \ref{cong2-iso-vx} (applied to $v$ together with any $k$-vertex induced subgraph of $H'$ which contains $u_v$) that $G$ contains a $k$-vertex subgraph $H$ with $\binom{k}{2} \not\equiv e(H) \pmod{2}$.  Thus from now on we may assume that $k$ is odd.

If in fact there exists $v \in V(G) \setminus V(H')$ and $w \in V(H')$ such that $vw \in E(G)$ then, considering $v$ together with any $k$-vertex induced subgraph of $H'$ containing both $u_v$ and $w$, it follows from Proposition \ref{even-odd-vertex} that $G$ contains a $k$-vertex subgraph $H$ with $\binom{k}{2} \not\equiv e(H) \pmod{2}$.  Thus we may assume from now on that for all $v \in V(G) \setminus V(H')$, $v$ has no neighbour in $H'$.  

We will show that in this case the second condition must hold, arguing in this case that for every $v_1, v_2 \in V(G) \setminus V(H')$ with $v_1 \neq v_2$ we have $v_1v_2 \in E(G)$.  Suppose, for a contradiction, that there exist non-adjacent $v_1$ and $v_2$ in $V(G) \setminus V(H')$.  But then the subgraph induced by $v_1$ and $v_2$ together with any $k-2$ vertices of $H'$ will have $\binom{k}{2} - (2k - 3) \not\equiv \binom{k}{2} \pmod{2}$ edges.  This completes the proof that if $G$ contains no even $k$-vertex subgraph then one of the properties 1 and 2 must hold.

Conversely, suppose that one of these two conditions holds.  If $G$ is a clique, it is trivial that every $k$-vertex induced subgraph of $G$ has precisely $\binom{k}{2}$ edges.  So suppose that $k$ is odd and that $G$ is the disjoint union of two cliques, $G_1$ and $G_2$.  Let $H$ be a $k$-vertex subgraph of $G$, with $|V(H) \cap V(G_1)| = i$.  Then the number of edges in $H$ is precisely $\binom{k}{2} - i(k-i)$.  Note that, as $k$ is odd, exactly one of $i$ and $k-i$ must be even, and so $i(k-i)$ is even, implying that $e(H) \equiv \binom{k}{2} \pmod{2}$, as required.
\end{proof}

This implies a characterisation of those sufficiently large graphs which contain no even $k$-vertex subgraph.

\begin{cor}
Let $G$ be a graph on $n \geq 2^{2k}$ vertices, where $k \geq 2$.  Then $G$ contains no even $k$-vertex subgraph if and only if $\binom{k}{2}$ is odd and either
\begin{enumerate}
\item $G$ is a clique, or
\item $k \equiv 3\pmod{4}$ and $G$ is the disjoint union of two cliques.
\end{enumerate}
\label{existence-even}
\end{cor}
\begin{proof}
Observe first that, by Theorem \ref{ramsey}, $G$ must contain either a clique or independent set on $k$ vertices.  If $\binom{k}{2}$ is even this is enough to guarantee the existence of a $k$-vertex even subgraph and hence to prove the result; if $\binom{k}{2}$ is odd (so $k \equiv 2 \pmod{4}$ or $k \equiv 3 \pmod{4}$) then $G$ contains an even $k$-vertex subgraph if and only if the graph contains a $k$-vertex subgraph $H$ such that $e(H) \not\equiv \binom{k}{2} \pmod{2}$.  The result then follows immediately from Lemma \ref{different-parity}.
\end{proof}

Similarly, we can completely characterise those sufficiently large graphs which contain no odd $k$-vertex subgraph.

\begin{cor}
Let $G$ be a graph on $n \geq 2^{2k}$ vertices, where $k \geq 2$.  Then $G$ contains no odd $k$-vertex subgraph if and only if one of the following conditions holds.
\begin{enumerate}
\item $G$ is an independent set.
\item $k$ is odd and $G$ is a complete bipartite graph.
\item $\binom{k}{2}$ is even and $G$ is a clique.
\item $k \equiv 1 \pmod{4}$ and $G$ is the disjoint union of two cliques.
\end{enumerate}
\label{existence-odd}
\end{cor}
\begin{proof}
It is straightforward to verify that any of the four conditions is sufficient to guarantee that $G$ contains no odd $k$-vertex subgraph.  To prove the reverse implication, suppose that $G$ contains no odd $k$-vertex subgraph.

Once again, we begin with the observation that, by Theorem \ref{ramsey}, $G$ must contain either a clique or independent set on $k$ vertices.  We consider two cases, depending on whether $\binom{k}{2}$ is even or odd.  

Suppose first that $\binom{k}{2}$ is odd (so $k \equiv 2 \pmod{4}$ or $k \equiv 3 \pmod{4}$).  In this case, if $G$ contains a $k$-clique then we have a $k$-vertex odd subgraph, so we may assume that $G$ contains an independent set on $k$ vertices.  Thus $\overline{G}$, the complement of $G$, contains a $k$-vertex clique.  Moreover, since $\binom{k}{2}$ is odd, we see that $\overline{G}$ contains an even $k$-vertex subgraph if and only if $G$ contains an odd $k$-vertex subgraph.  It therefore follows from Lemma \ref{different-parity} that if $G$ contains no odd $k$-vertex subgraph then either $\overline{G}$ is a clique, implying that $G$ is an independent set, or else $k \equiv 3 \pmod{4}$ and $\overline{G}$ is the disjoint union of two cliques, in which case $G$ is a complete bipartite subgraph.

Now suppose that $\binom{k}{2}$ is even (so $k \equiv 0 \pmod{4}$ or $k \equiv 1 \pmod{4}$).  In this case, $\overline{G}$ contains an odd $k$-vertex subgraph if and only if there is an odd $k$-vertex subgraph in $G$.  We know from Theorem \ref{ramsey} that at least one of $G$ and $\overline{G}$ must contain a $k$-clique, and so we can apply Lemma \ref{different-parity} to the appropriate graph.  If $G$ contains a clique, then Lemma \ref{different-parity} tells us that if $G$ contains no odd $k$-vertex subgraph then either $G$ is a clique, or else $k \equiv 1 \pmod{4}$ and $G$ is the disjoint union of two cliques.  If, on the other hand, $\overline{G}$ contains a clique, Lemma \ref{different-parity} tells us in this case that if $\overline{G}$ contains no odd $k$-vertex subgraph then either $\overline{G}$ is a clique, in which case $G$ is an independent set, or else $k \equiv 1 \pmod{4}$ and $\overline{G}$ is the disjoint union of two cliques, in which case $k$ is odd and $G$ is a complete bipartite graph.
\end{proof}

We now prove the key lemma of this section, which demonstrates that any graph containing a large number of $k$-cliques and at least one $k$-vertex induced subgraph with a number of edges that differs in parity from $\binom{k}{2}$ must in fact contain a large number of such subgraphs.

\begin{lma}
Let $k \geq 3$, and let $G$ be a graph on $n$ vertices that contains at least $\frac{1}{2^{2k^2+1}}\binom{n}{k}$ $k$-vertex cliques.  Then either $G$ contains no $k$-vertex subgraph $\widetilde{H}$ with $e(\widetilde{H}) \not\equiv \binom{k}{2} \pmod{2}$, or else $G$ contains at least $\frac{1}{2^{2k^2+1}k^2n^2}\binom{n}{k}$ such subgraphs.
\label{many-parity}
\end{lma}
\begin{proof}
For any $A \subseteq [k]$, we say that a $k$-vertex clique $H$ in $G$ is \emph{$A$-replaceable} if there exist sets of vertices $X \subseteq V(G) \setminus V(H)$ and $Y \subseteq V(H)$ such that $|X| = |Y| \in A$ and $(V(H) \setminus Y) \cup X$ induces a subgraph $\widetilde{H}$ with $e(\widetilde{H}) \not\equiv \binom{k}{2} \pmod{2}$; we refer to this new subgraph $\widetilde{H}$ as a \emph{$A$-replacement} of $H$.

Suppose that every $k$-vertex clique in $G$ is $\{1,2\}$-replaceable.  Note that any even $k$-vertex subgraph can be a $\{1,2\}$-replacement of at most $k(n-k) + \binom{k}{2}\binom{n-k}{2} < k^2n^2$ distinct $k$-cliques.  Thus the total number of subgraphs $\widetilde{H}$ in $G$ with $e(\widetilde{H}) \not\equiv \binom{k}{2} \pmod{2}$ must be at least $\frac{1}{2^{2k^2+1}k^2n^2}\binom{n}{k}$, and so we are done.  Hence we may assume from now on that there is at least one $k$-vertex clique $H$ in $G$ that is not $\{1,2\}$-replaceable.  There are two cases to consider, depending on whether $k$ is even or odd.  

Suppose first that $k$ is even; we know from Corollary \ref{cong2-iso-vx} that in this case $H$ would be $\{1\}$-replaceable if there exists any vertex $v \in V(G) \setminus V(H)$ such that $N(v) \cap V(H) \subsetneq V(H)$, so we may assume that every vertex $v \in V(G) \setminus V(H)$ is adjacent to every vertex in $H$.  If there exist two vertices $u,w \in V(G) \setminus V(H)$ such that $uw \notin E(G)$, the subgraph induced by $u$ and $w$ together with any $k-2$ vertices of $H$ would then have exactly $\binom{k}{2} - 1$ edges, implying that $H$ is $\{2\}$-replaceable.  Hence, as we are assuming that $H$ is not $\{1,2\}$-replaceable, we see that $G$ must in fact be a clique; so $G$ contains no $k$-vertex subgraph $\widetilde{H}$ with $e(\widetilde{H}) \not\equiv \binom{k}{2} \pmod{2}$, and we are done.

Now suppose that $k$ is odd.  In this case we know from Lemma \ref{even-odd-vertex} that every vertex $v \in V(G)\setminus V(H)$ must either be adjacent to every vertex in $H$, or else have no neighbour in $H$, otherwise $H$ would be $\{1\}$-replaceable.  Let $U$ be the set of vertices in $V(G) \setminus V(H)$ that are adjacent to every vertex in $H$, and $W$ the set of vertices in $V(G) \setminus V(H)$ that have no neighbour in $H$ (so $U$ and $W$ partition $V(G) \setminus V(H)$).  We claim that each of $U$ and $W$ must induce a clique.  First suppose that there is a pair of nonadjacent vertices $u_1,u_2 \in U$.  Then the subgraph induced by $u_1$ and $u_2$ together with any $k-2$ vertices of $H$ would have precisely $\binom{k}{2} - 1 \not\equiv \binom{k}{2} \pmod{2}$ edges, implying that $H$ is $\{2\}$-replaceable.  Similarly, if there is a pair of nonadjacent vertices $w_1,w_2 \in W$ then the subgraph induced by $w_1$ and $w_2$ together with any $k-2$ vertices of $H$ would have $\binom{k}{2} - 2k + 3 \not\equiv \binom{k}{2} \pmod{2}$ edges, again implying that $H$ is $\{2\}$-replaceable.  Thus we may indeed assume that $U$ and $W$ each induce a clique.  Set $U' = U \cup V(H)$, and observe that $U'$ also induces a clique in $G$.

Suppose that $G$ does contain a $k$-vertex subgraph $\widetilde{H}$ such that $e(\widetilde{H}) \not\equiv \binom{k}{2} \pmod{2}$; we will argue that in this case we must actually have at least $\frac{1}{2^{2k^2+1}k^2n^2}\binom{n}{k}$ such subgraphs.  We know from Lemma \ref{different-parity} that, as $k$ is odd, this assumption implies that $G$ can be neither a clique nor the disjoint union of two cliques.  This implies that there exists at least one edge $ab$ with $a \in U'$ and $b \in W$, and at least one non-edge $xy$ with $x \in U'$ and $y \in W$.  This situation is illustrated in Figure \ref{cross-edges}.  The remainder of the argument treats $U'$ and $W$ symmetrically so, as $U'$ and $W$ partition $V(G)$, we may assume without loss of generality that $|W| \geq n/2$.  Set $W_a = N(a) \cap W$, $W_x = N(x) \cap W$, $\overline{W_a} = W \setminus W_a$ and $\overline{W_x} = W \setminus W_x$.

\begin{figure}[h]
\centering
\includegraphics[width=0.2 \linewidth]{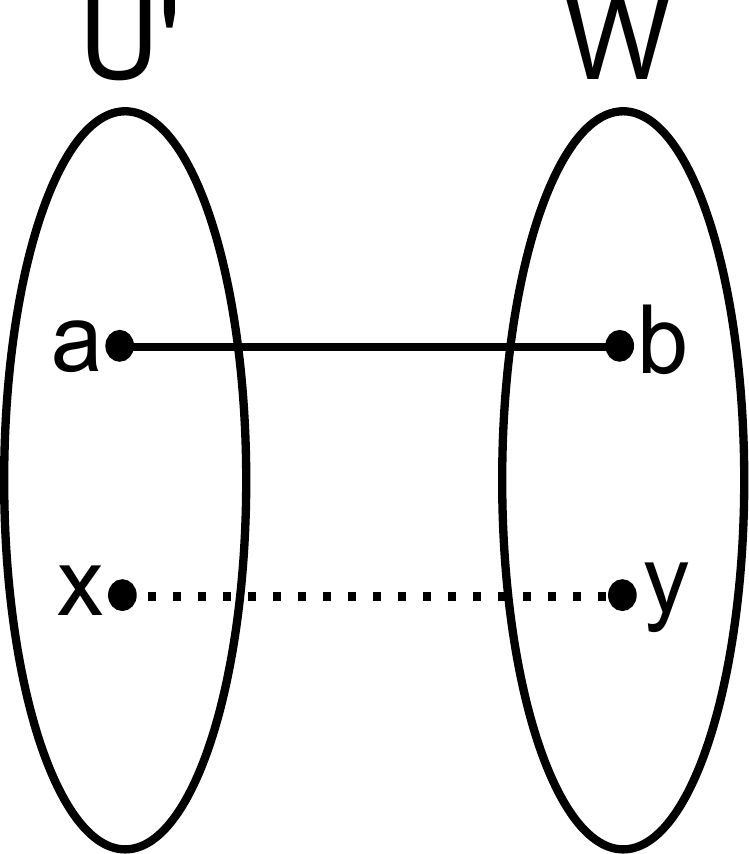}
\caption{There must be at least one edge and at least one non-edge between $U'$ and $W$.}
\label{cross-edges}
\end{figure}

Consider first the case that $|W_x| \geq n/6$.  In this case the subgraph induced by $x$ and $y$ together with any $k - 2$ vertices from $W_x$ contains all possible edges apart from $xy$, and so has exactly $\binom{k}{2} - 1$ edges.  Thus $G$ contains at least 
\begin{align*}
\binom{\frac{n}{6}}{k-2} & = \frac{\binom{\frac{n}{6}}{k-2}}{\binom{n}{k}} \binom{n}{k} \\
						 & \geq \frac{\left( \frac{n}{6(k-2)} \right) ^{k-2}}{\left(\frac{en}{k}\right)^k} \binom{n}{k} \\
						 & = \frac{k^k}{e^k 6^{k-2} (k-2)^{k-2} n^2} \binom{n}{k} \\
						 & > \frac{1}{(6e)^kn^2} \binom{n}{k} \\
						 & > \frac{1}{2^{2k^2+1}k^2n^2}\binom{n}{k}
\end{align*}
$k$-vertex subgraphs whose number of edges differs in parity from $\binom{k}{2}$.

Now suppose instead that $|\overline{W_a}| \geq n/6$. Then the subgraph induced by $a$ and $b$ together with any $k-2$ vertices from $\overline{W_a}$ contains all possible edges incident with vertices other than $a$, and is missing precisely $k-2$ possible edges incident with $a$.  Thus any such subgraph has exactly $\binom{k}{2} - (k-2)$ edges which, as $k$ is odd, must differ in parity from $\binom{k}{2}$.  Once again, therefore, we see that $G$ contains at least 
$$\binom{\frac{n}{6}}{k-2} > \frac{1}{2^{2k^2+1}k^2n^2}\binom{n}{k}$$
$k$-vertex subgraphs whose number of edges differs in parity from $\binom{k}{2}$.

It must therefore be that $|\overline{W_x} \cap W_a| \geq n/6$; observe that this implies that $a \neq x$, as otherwise $\overline{W_x} \cap W_a = \emptyset$.  But in this case the subgraph induced by $a$ and $x$ together with any $k-2$ vertices from $\overline{W_x} \cap W_a$ contains all possible edges not incident with $x$, and is missing precisely $k-2$ possible edges incident with $x$ (as the only neighbour of $x$ is $a$), so contains exactly $\binom{k}{2} - (k-2) \not\equiv \binom{k}{2} \pmod{2}$ edges.  Thus in this case $G$ must still contain at least
$$\binom{\frac{n}{6}}{k-2} > \frac{1}{2^{2k^2+1}k^2n^2}\binom{k}{2}$$
$k$-vertex subgraphs whose number of edges differs in parity from $\binom{k}{2}$.

Hence we see that, if $G$ contains at least one $k$-vertex subgraph whose number of edges differs in parity from $\binom{k}{2}$, it must in fact contain at least $\frac{1}{2^{2k^2+1}k^2n^2}\binom{n}{k}$ such subgraphs, completing the proof.
\end{proof}

We can apply the previous result to demonstrate that any sufficiently large graph either contains no even $k$-vertex subgraph or else must contain a large number of even $k$-vertex subgraphs.

\begin{thm}
Let $k \geq 3$ and let $G$ be a graph on $n \geq 2^{2k}$ vertices.  Then either $G$ contains no even $k$-vertex subgraph or else $G$ contains at least 
$$\frac{1}{2^{2k^2+1}k^2n^2}\binom{n}{k}$$
even $k$-vertex subgraphs.
\label{even-density}
\end{thm}
\begin{proof}
Observe first that, by Corollary \ref{ramsey-cor} and equation \eqref{ramsey-cor-bound}, $G$ must contain at least $\frac{1}{2^{2k}}\binom{n}{k}$ $k$-vertex subsets that induce either cliques or independent sets.  If $\binom{k}{2}$ is even, any such set will induce a $k$-vertex even subgraph, and so we are done; thus we may assume from now on that $\binom{k}{2}$ is odd (so $k \equiv 2 \pmod{4}$ or $k \equiv 3 \pmod{4}$).  Any $k$-vertex independent set still has an even number of edges, so if there are at least $\frac{1}{2^{2k^2+1}}\binom{n}{k}$ $k$-vertex independent sets in $G$ then we are done.  We may assume, therefore, that $G$ contains at least $\frac{1}{2^{2k^2+1}}\binom{n}{k}$ $k$-cliques.  The result now follows immediately from Lemma \ref{many-parity}.
\end{proof}

We now prove a corresponding result for the case of odd $k$-vertex subgraphs.

\begin{thm}
Let $k \geq 3$ and let $G$ be a graph on $n \geq 2^{2k}$ vertices.  Then either $G$ contains no odd $k$-vertex subgraph or else $G$ contains at least 
$$\frac{1}{2^{2k^2+1}k^2n^2}\binom{n}{k}$$
odd $k$-vertex subgraphs.
\label{odd-density}
\end{thm}
\begin{proof}
Once again, we begin with the observation that, by Corollary \ref{ramsey-cor} and equation \eqref{ramsey-cor-bound}, $G$ must contain at least $\frac{1}{2^{2k}}\binom{n}{k}$ $k$-vertex subsets that induce either cliques or independent sets; in particular this implies that at least one of $G$ and $\overline{G}$ contains at least $\frac{1}{2^{2k^2+1}}\binom{n}{k}$ $k$-cliques.  There are two cases to consider, depending on whether $\binom{k}{2}$ is even or odd.

Suppose first that $\binom{k}{2}$ is even.  In this case, it is clear that any subset $U \in V^{(k)}$ induces an odd subgraph in $G$ if and only if $U$ also induces an odd subgraph in $\overline{G}$, so the number of odd subgraphs in $G$ and $\overline{G}$ is equal.  We know that one of these graphs contains at least $> \frac{1}{2^{2k^2+1}}\binom{n}{k}$ $k$-cliques; without loss of generality we may assume that this is $G$.  The result now follows immediately from Lemma \ref{many-parity}.

Now suppose that $\binom{k}{2}$ is odd.  Thus, if $G$ contains at least $\frac{1}{2^{2k^2+1}}\binom{n}{k}$ $k$-cliques we are done immediately, so we may assume instead that $\overline{G}$ contains at least this many $k$-cliques.  Lemma \ref{many-parity} therefore tells us that $\overline{G}$ either contains no even $k$-vertex subgraph or else contains at least $\frac{1}{2^{2k^2+1}k^2n^2}\binom{n}{k}$ even $k$-vertex subgraphs.  Observe that, when $\binom{k}{2}$ is odd, there is a one-to-one correspondence between even $k$-vertex subgraphs in $\overline{G}$ and odd $k$-vertex subgraphs in $G$, so this implies that $G$ either contains no odd $k$-vertex subgraph or else contains at least $\frac{1}{2^{2k^2+1}k^2n^2}\binom{n}{k}$ odd $k$-vertex subgraphs, as required.
\end{proof}

\subsection{Algorithmic implications}
\label{alg-imp}

In this section, we make use of the structural results of Section \ref{structural} above to derive algorithms to decide the existence of $k$-vertex even or odd subgraphs, and to approximate the number of such subgraphs.

We begin by showing that \paramdec{Even Subgraph} is in FPT.

\begin{thm}
\paramdec{Even Subgraph} can be solved in time $O(\binom{2^{2k}}{k}\binom{k}{2} + n^2)$.
\label{even-decision}
\end{thm}
\begin{proof}
Consider the following algorithm to determine whether a graph $G$ on $n$ vertices contains a $k$-vertex even subgraph.
\begin{enumerate}
\item \label{k=1} If $k=1$, return ``YES''.
\item \label{exhaustive} If $n < 2^{2k}$, perform an exhaustive search to determine whether $G$ contains a $k$-vertex even subgraph, and return the answer.
\item \label{clique-even} If $k \equiv 0 \pmod{4}$ or $k \equiv 1 \pmod{4}$, return ``YES''.
\item \label{clique-test} If $G$ is a clique, return ``NO''.
\item \label{k-even} If $k \equiv 2 \pmod{4}$, return ``YES''.
\item \label{disjoint-clique-test} If $G$ is a disjoint union of two cliques, return ``NO''.
\item \label{finish} Return ``YES''.
\end{enumerate}
The correctness of this algorithm follows immediately from Corollary \ref{existence-even}.  Moreover, it is clear that step \ref{exhaustive} can be performed in time $O(\binom{2^{2k}}{k}\binom{k}{2})$, and that steps \ref{clique-test} and \ref{disjoint-clique-test} can each be performed in time $O(n^2)$, while steps \ref{k=1}, \ref{clique-even}, \ref{k-even} and \ref{finish} each take time at most $O(k)$.  Thus this algorithm runs in time $O(\binom{2^{2k}}{k}\binom{k}{2} + n^2)$, as required.
\end{proof}

We now prove a corresponding result for \paramdec{Odd Subgraph}.

\begin{thm}
\paramdec{Odd Subgraph} can be solved in time $O(\binom{2^{2k}}{k}\binom{k}{2} + n^2)$.
\label{odd-decision}
\end{thm}
\begin{proof}
The proof proceeds along very much the same lines as that of Theorem \ref{even-decision} above.  Consider the following algorithm:
\begin{enumerate}
\item If $k = 1$, return ``NO''.
\item If $n < 2^{2k}$, perform an exhaustive search to determine whether $G$ contains a $k$-vertex odd subgraph, and return the answer.
\item If $G$ is an independent set, return ``NO''.
\item If $k$ is odd and $G$ is a complete bipartite graph, return ``NO''.
\item If $k \equiv 0 \pmod{4}$ or $k \equiv 1 \pmod{4}$ and $G$ is a clique, return ``NO''.
\item If $k \equiv 1 \pmod{4}$ and $G$ is a disjoint union of two cliques, return ``NO''.
\item Return ``YES''.
\end{enumerate}
In this case correctness follows from Corollary \ref{existence-odd}, while it is straightforward to verify that each of the steps may be performed within the permitted time.
\end{proof}

Finally, we show that both \paramcount{Even Subgraph} and \paramcount{Odd Subgraph} are efficiently approximable.

\begin{thm}
There exists an FPTRAS for \paramcount{Even Subgraph}, and also for \paramcount{Odd Subgraph}.
\end{thm}
\begin{proof}
We know by Theorem \ref{even-density} that, if $G$ contains at least one even $k$-vertex subgraph, then $G$ contains at least $\frac{1}{2^{2k^2+1}k^2n^2}\binom{n}{k}$ even $k$-vertex subgraphs. It therefore follows from Lemma \ref{exists-fptras} that there is a randomised algorithm which, for every $\epsilon > 0$ and $\delta \in (0,1)$, outputs an integer $\alpha$, such that, if $N$ denotes the number of even $k$-vertex subgraphs in $G$,
$$\mathbb{P}[|\alpha - N| \leq \epsilon \cdot N] \geq 1 - \delta,$$
taking time at most $g(k)q(n,\epsilon^{-1},\log(\delta^{-1}))$, where $g$ is a computable function and $q$ is a polynomial.  The existence of an  FPTRAS for \paramcount{Even Subgraph} follows immediately.

The argument for \paramcount{Odd Subgraph} proceeds in exactly the same way, using Theorem \ref{odd-density}.
\end{proof}

\section{Conclusions and open problems}

We have shown that the parameterised subgraph counting problems \paramcount{Even Subgraph} and \paramcount{Odd Subgraph} are both \#W[1]-complete when parameterised by the size of the desired subgraph; in fact we prove hardness for a more general family of parameterised subgraph counting problems which generalises both of these specific problems.  This intractability result complements several recent hardness results for parameterised subgraph counting problems, making further progress towards a complete complexity classification of this type of parameterised counting problem.

On the other hand, we show that both \paramdec{Even Subgraph} and \paramdec{Odd Subgraph} are in FPT, and moreover that the counting problems \paramcount{Even Subgraph} and \paramcount{Odd Subgraph} each admit a FPTRAS, and so are efficiently approximable from the point of view of parameterised complexity.  The approximability proofs rely on some surprising structural results which show that a sufficiently large graph $G$ either contains no even (respectively odd) induced $k$-vertex subgraph, or else must contain a large number of such subgraphs (specifically, at least a $\frac{1}{f(k)n^2}$ proportion of $k$-vertex induced subgraphs must have the desired edge parity, where $f$ is an explicit computable function).

A natural question arising from this work is whether the same results (either hardness of exact counting, or the existence of a FPTRAS) hold for other subgraph counting problems in which the desired property depends only on the number of edges in the subgraph.  As a first step, it would be interesting to consider the problem of counting induced $k$-vertex subgraphs having a specified number of edges modulo $p$, for any fixed natural number $p > 2$.

\bibliographystyle{amsplain}
\bibliography{../param_counting_refs} 

\end{document}